\documentclass[11pt]{amsart}

\usepackage{amssymb,amsthm,amsmath}
\usepackage[numbers,sort&compress]{natbib}
\usepackage{color}
\usepackage{graphicx}
\usepackage{tikz}
\usepackage{amssymb,amsthm,amsmath}
\usepackage[numbers,sort&compress]{natbib}
\usepackage{color}
\usepackage{graphicx}
\usepackage{tikz}

\title[ ]{  Noncompact complete Riemannian manifolds with  singular continuous spectrum embedded into the
essential spectrum of the Laplacian, I. The hyperbolic case. }
\author{ Svetlana Jitomirskaya}
\address[ Svetlana Jitomirskaya]{ Department of Mathematics, University of California, Irvine, CA 92697-3875, USA}
\email{szhitomi@math.uci.edu}

\author{Wencai Liu}
\address[Wencai Liu]{Department of Mathematics, University of California, Irvine, CA 92697-3875, USA}\email{liuwencai1226@gmail.com}
\address{Current address: Department of Mathematics, Texas A\&M University, College Station, TX 77843-3368, USA}

\theoremstyle{plain}
\newtheorem{theorem}{Theorem}[section]

\newtheorem{lemma}[theorem]{Lemma}

\newcommand{\R}{\mathbb{R}}

\theoremstyle{definition}

\newtheorem{remark}[theorem]{Remark}

\begin{document}


\begin{abstract}
We construct Riemannian manifolds with singular continuous spectrum
 embedded in the absolutely continuous spectrum of the Laplacian. Our
 manifolds are  asymptotically
 hyperbolic with
 sharp curvature bounds.

\end{abstract}
\maketitle
\section{Introduction and main results}

Let $(M_n, g), \, n\geq 2,$ be an $n$-dimensional connected  noncompact
complete Riemannian manifold.
The Laplace-Beltrami operator  $\Delta:=\Delta_g$ on   $M:=(M_n,g)$, is essentially self-adjoint on $C^{\infty}_0(M_n)$.  We also denote by $\Delta$ its unique self-adjoint extension to
$L^2(M_n,dv_g)$. 

We refer the readers to \cite{donn} for a review of results on the
spectral theory of Laplacians on non-compact manifolds. Most of the
past work has been focused on proofs of the purity of absolutely
 continuous spectrum, guaranteed by the asymptotic curvature
 conditions, going back to
 \cite{pinsky1979spectrum,donnelly1990negativ}. Several extensions of
 purity results
 have also  appeared recently \cite{Liu1,Liu2,ito1,ito2}. Lately,
 some attention has turned to the opposite phenomenon.
 Kumura \cite{kum} constructed manifolds with an eigenvalue
 embedded in the spectrum of the Laplacian.
 In \cite{jl} we constructed 
manifolds with arbitrary finite or countable
 subset of the essential spectrum embedded as eigenvalues.
This brings a natural question whether singular
 continuous spectrum can also be embedded in the essential (absolutely
 continuous) spectrum of the
 Laplace-Beltrami operator. The
 goal of this paper is to construct such manifolds.
We prove

\begin{theorem}\label{Thm1}
For $K_0<0,$ there exist  smooth simply  connected $n$-dimensional   Riemannian manifolds
such that
\begin{enumerate}
   \item
   $\sigma_{{\rm ess}}(-\Delta)=\sigma_{{\rm ac}}(-\Delta)= \left[\frac{|K_0|}{4}(n-1)^2,\infty \right)$,
   \item $\sigma_{{\rm sc}}(-\Delta)\neq \emptyset.$\footnote{It is
       then automatically embedded in $\left[\frac{|K_0|}{4}(n-1)^2,\infty \right)$.}
\end{enumerate}
\end{theorem}


Despite a significant interest in the Schr\"odinger operators
community in the last 30 years and various ubiquity results
(initiated by \cite{wonderland}) singular continuous spectrum remains rather
mysterious in the spectral theory and has been virtually unseen and unstudied in
spectral geometry. In particular, to the best of our knowledge, there have been no previous
constructions of Riemannian manifolds with embedded singular
continuous spectrum of the Laplacian. The only appearance of the
singular continuos spectrum in the context of Laplace-Beltrami
operators we are aware of  is \cite{simongen} where Simon proved
topological genericity of manifolds with purely singular continuous
spectrum in a class of metrics on the $2d$ infinite cylinder (so not
simply connected, and
without an explicit construction).

Singular continuous spectral measures are supported on zero measure
sets, yet give zero weight to every point, making them particularly
difficult to control explicitly. Quite often singular continuous spectrum is
proved by ruling out existence of absolutely continuous and point
components (or turning the reasoning above on its head as in \cite{wonderland}). Clearly, this is not going to work for
singular continuous spectrum  {\it embedded} into absolutely
continuous one, making corresponding questions
especially hard.

In this paper we study the asymptotically hyperbolic case: Riemannian manifolds with
the radial curvature $K_{\rm rad}(r)$ (sectional curvature with one
fixed direction $\nabla r$) approaching $K_0<0$ as $r\to\infty.$ If
$K_{\rm rad}(r)=K_0<0$ is constant,  it is well known that $\sigma(-\Delta )=\sigma_{\rm
  ac}(-\Delta_g)=\sigma_{{\rm ess}}(-\Delta_g)=
\left[\frac{|K_0|}{4}(n-1)^2,\infty  \right)$  and the singular spectrum (the
union of point and singular continuous spectra) is empty. The
essential spectrum is preserved under decaying perturbations and it is
natural to expect that no embedded singular spectrum will persist when $K_{\rm rad}(r)$ approaches  $K_0$
sufficiently fast, but point and singular continuous spectra can  be
embedded   into the
essential  (absolutely continuous) spectrum for slower rates of decay
of  $|K_{\rm rad}(r)-K_0|.$ Note that compact perturbations of constant curvature can only
lead to eigenvalues below the essential spectrum, so embedding
questions are naturally tied to the rate of decay. Sharp decay thresholds have been established
for existence of metrics with an embedded eigenvalue \cite{kum} (see
also \cite{Liu1} for a simple proof of sharpness) and
with an embedded arbitrary countable (in particular, dense) set
\cite{jl} (also for the flat, i.e. $K_0=0,$ case). Here we prove a correspondingly  more precise
version of Theorem \ref{Thm1}

\begin{theorem}\label{Thmmanifold}
 Suppose $K_0< 0$. Let $h(r)>0$ be   any 
function  on  $(0,\infty)$  with $  \lim_{r\to \infty}h(r) =
\infty$. Then  there exist  smooth simply  connected  
Riemannian manifolds
 $(M_n,g)$
such that
\begin{enumerate}
\item   $ |K_{{\rm rad}}(r)-K_0|=  O\left(\frac{h(r)}{1+r}\right)$,
   \item
   $\sigma_{{\rm ess}}(-\Delta)=\sigma_{{\rm ac}}(-\Delta)= \left[\frac{|K_0|}{4}(n-1)^2,\infty \right)$,
   \item $\sigma_{{\rm sc}}(-\Delta)\neq \emptyset.$
\end{enumerate}
 \end{theorem}
 \begin{remark}
 Modifying  our construction, the spectral measure of the Laplacian
 can have  both  pure point  and singular continuous components on $\left(\frac{|K_0|}{4}(n-1)^2,\infty\right) $.
 \end{remark}

We expect that our result is sharp, that is, like in the 1D case discussed below, $\frac{O(1)}{1+r}$ provides
threshold for existence of asymptotically hyperbolic metrics with
embedded singular continuous spectrum: for manifolds with
$|K_{\rm rad}(r)-K_0|<C(1+r)^{-1},$ the essential spectrum should be purely absolutely
continuous. So far it has been established under somewhat more restrictive
conditions. Kumura     \cite{kumura2013limiting} proved absolute
continuity of the  Laplacian by the limiting absorption principle ( originally from Agmon's theory \cite{agmon1975})
under the condition $|K_{\rm
  rad}(r)-K_0|=\frac{O(1)}{r^{1+\delta}}, \delta>0$  and assuming convexity of the Hessian of $r$.
Donnelly  used  exhaustion function to investigate the spectral
structure of the Laplacian,
which can also show the absence of singular continuous spectrum for some manifolds \cite{donnelly1999,donnelly1997exhaustion}.

There is a remarkable similarity between results on curvature  thresholds
for embedded eigenvalues for the non-compact manifolds in arbitrary
dimension  and for 1D Schr\"odinger operators
with decaying potentials. This leads to a natural conjecture that the
curvature threshold for existence of metrics with embedded singular
continuous spectrum is also  going to be the same as in the 1D Schr\"odinger
case, where this was a known difficult problem, popularized by B. Simon in the 90s and
included in his list of 15 Schr\"odinger  operator problems for the
XXI century \cite{XXI}. Unlike for the manifolds, for Schr\"odinger
operators, existence of {\it some} potentials with prescribed spectral
behavior is guaranteed by the inverse spectral theory \cite{levitan,marchenko}, so the issue is potentials with certain decay. Existence
of $L^2$ potentials with embedded singular continuous spectrum was
proved by Denisov \cite{den} and followed from Killip-Simon's
criterion \cite{ks} but in an implicit way. Potentials with power
decaying solutions on a set of expected Hausdorff dimension
\cite{ck,rem1} were constructed by Remling \cite{remduke,krsharp}, but this
was insufficient to infer existence of embedded singular continuous
component. Decaying potentials with purely
singular continuous spectrum were constructed in \cite{rem2,kls}.
 An explicit construction of potential that has singular continuous
 spectrum {\it embedded} into absolutely continuous
(and a sharp result in terms of decay) was given by Kiselev
\cite{kiselev2005imbedded}, therefore solving Simon's problem. He proved  that if the potential $V(x)=\frac{O(1)}{1+x}$, then the singular continuous spectrum  of $-D^2+V$ is empty,
but given any positive function $h(x) $ tending to infinity as $x$ grows, there exist potentials $V(x)$ such that $|V(x)|\leq \frac{h(x)}{1+x}$ and the operator $-D^2+V$ has a non-empty singular continuous spectrum on $(0,\infty)$\cite{kiselev2005imbedded}.
By Weyl theorem and classical results in  \cite{christ1998absolutely,remling1998absolutely,deift1999absolutely},
both the essential spectrum and absolutely continuous spectrum of
$-D^2+V$ constructed by Kiselev are $(0,\infty)$.

In this paper, we use Kiselev's potentials to construct our
manifolds. The Riemannian manifolds $(M,g)$ we construct  are rotationally
symmetric, and we effectively reduce the problem to a one-dimensional
Schr\"odinger operator, with the main work needed to guarantee the
existence of smooth metrics leading to a 1D potential with desired
properties. It turns out this is possible to do in the asymptotically
hyperbolic case, using Kiselev's construction almost as a black box. The asymptotically flat case (i.e. $K_0=0$) however turns out to be
more difficult, with corresponding problem unsolvable without further
assumptions on the potential, thus requiring to significantly modify Kiselev's
construction to guarantee the additional desired structure of the
potential. This will be done in \cite{jl2}.

To construct a rotationally symmetric manifold, we fix some $O\in M_n$
as the origin.  Using the radial coordinates (from
$O$) we construct  Riemannian manifold with the structure of the form
$$(M_n,g) =
\bigl( {  \R}^n, dr^2 + f_1^2(r) g_{S^{n-1}(1)} \bigr),$$
 where $
g_{S^{n-1}(1)} $ is the standard Riemannian metric on the unit
sphere, and we need to construct $f_1$ so that the Laplacian has the
desired properties.
To determine the
spectral representation of the Laplacian on a rotationally symmetric
manifolds, one can use separation of variables .

 Let $Y_{i,j}(\theta)$,  $\theta\in S^{n-1}$,  $i\geq 0$ and  $j=1,2,\cdots,q_i$, be the  spherical harmonics. They   form a complete orthonormal
basis for $L^2(S^{n-1})$ \cite{Modern}.
Each $Y_{i,j}(\theta)$ belongs to a $q_i$ dimensional eigenspace of the
spherical Laplacian with corresponding eigenvalue $\lambda_i$. One may expand $ \phi(r,\theta)\in L^2(M_n,g)$ as
\begin{equation*}
    \phi(r,\theta)=\sum_{i=0}^{\infty}\sum_{j=1}^{q_i}\phi_{i,j}(r)Y_{i,j}(\theta).
\end{equation*}
A computation  gives
\begin{equation*}
    -\Delta \phi=\sum_{i=0}^{\infty}\sum_{j=1}^{q_i} (-\Delta_i)\phi_{i,j}(r)Y_{i,j}(\theta),
\end{equation*}
where $-\Delta_i$ is defined on $L^2(\R^+,f_1^{n-1}dr)$, by
\begin{align}\label{0lap1}
    - \Delta_i    v
    = -\left\{ \frac{\partial ^2}{\partial r^2}
    + (n-1) \frac{f_1^\prime(r)}{f_1(r)} \frac{\partial }{\partial r} \right\} v+\frac{\lambda_i}{f_1^2} v.
\end{align}
Notice that   $v(r)$ is a function on $M$ only depending on the radius $r$.
Thus $\Delta$ is decomposed into a
direct sum of one-dimensional  operators  $\Delta_i$ with multiplicity $q_i$.

We now renormalize the measure to Lebesgue.
Let $U(v)=f_1^{\frac{n-1}{2}} v$ and  $$ L_i=U(-\Delta_i) U^{-1}.$$
$U$ is clearly unitary, making
$ -\Delta_i$ on $L^2((0,\infty),f_1^{n-1}dr)$   unitarily equivalent to operator  $L_i$ on $L^2((0,\infty), dr)$. Straightforward calculations give
\begin{equation}\label{Gop}
     L_iu=-   D^2u +V_iu
\end{equation}
where
\begin{equation}\label{Gp}
   D^2u=u^{\prime\prime}, V_i =\frac{(n-1)(n-3)}{4}\left(\frac{f_1^{\prime}}{f_1}\right)^2+\frac{n-1}{2}\frac{f_1^{\prime\prime}}{f_1}+\frac{\lambda_i}{f_1^2}.
\end{equation}

The proof now almost reduces to showing the existence of a singular
continuous component for some $L_i$, which is
a one-dimensional problem.
However,  in order to make the  manifold smooth in the neighborhood of $O$,  $ f_1^{\text{even}}$ must vanish at $0$ and one must have $f_1^{\prime}(0)\neq 0 $.
This makes  $ \frac{f_1^{\prime}}{f_1}(r)$  and $V(r)$
singular at the point $r=0$, so we need to deal with one-dimensional
Schr\"odinger operator (\ref{0lap1}) or \eqref{Gop} with singularities
at both $0$ and $\infty$.

It is well known that we have
\begin{equation}\label{Gk}
     K_{\rm rad}(r)    =   -\frac{f_1''(r)}{f_1(r)}.
\end{equation}

 Our goal therefore is to construct
$f_1(r)$ such that
the one-dimensional Schr\"odinger operator given by  \eqref{Gop} has  non-empty singular continuous spectrum, and
the radial curvature \eqref{Gk} eventually satisfies
\begin{equation}\label{Gk1}
    |K_{\rm rad}(r)-K_0|\leq \frac{h(r)}{1+r}.
\end{equation}

Here is   the sketch of our construction.

In the neighborhood of $O$( i.e., $r=0$), we use  Euclidean metric. Then the Schr\"odinger  operator \eqref{Gop} is limit point at the left  singular point $r=0$.
For the Euclidean space, the spectral analysis can proceed by  the
generalized eigen-expansion, which is well known for the Hankel transformation (Bessel type functions).
Our first step is to obtain similar results by  the generalized eigen-expansion for
$-D+V$ where $V$ is generated by  the  Euclidean  metrics  only for small
values of $r$.

For large $r$ (neighborhood of $r=\infty$), we will adapt Kiselev's
construction \cite{kiselev2005imbedded}, which originally  was done for a
Schr\"odinger operator without a singular point at $r=0$.
There are two difficulties  here. First,   we need to construct $f_1$
such that the 1D potential given by \eqref{Gp} is what one gets from the Kiselev's construction
and the radial curvature given by \eqref{Gk} satisfies
\eqref{Gk1}. It is this step that becomes impossible in the
asymptotically flat case without further requirements on the 1D
potential. Second, $f_1$ constructed here for large $r$ should ``match"  $f_1$ in the neighborhood
of $r=0$  so that we can use the generalized eigen-expansion to complete the spectral analysis.

The rest of the paper is organized as follows:
In \S 2, we  set up the spectral analysis of Bessel type potentials.
In \S 3, we  give all the remaining technical preparations.
In \S 4,  we complete the proof of  Theorem \ref{Thmmanifold}.

 \section{Spectral analysis of Bessel type potentials}

 As   mentioned in the introduction,  for $r<1$,  we define the metric
 to be Euclidean, that is
 $f_1=r$.
 Thus, for $r<1,$ the potential  $V_i$  given by \eqref{Gp} is
 \begin{equation}\label{Pbessel}
    V_i(r)=\left(\frac{(n-1)(n-3)}{4}+\lambda_i\right)\frac{1}{r^2},
 \end{equation}

 By the fact that $\lambda_i\to \infty$, we can choose
some $i$ so that
 \begin{equation}
    \frac{(n-1)(n-3)}{4}+\lambda_i\geq 1,
 \end{equation}
 In the followin,g we fix such $\lambda_i$ and
 let
 \begin{equation}\label{Gmay212}
    \nu^2= \frac{(n-1)(n-3)}{4}+\lambda_i+\frac{1}{4}
 \end{equation}
 so that $\nu>1$.
 Now we only consider the operator $L_i$  on  $L^2(\R^+,dr)$. We omit the dependence on $i$ for simplicity.

Thus we have
 \begin{equation}\label{Gop0}
     L u=-   D^2u +Vu
\end{equation}
and by \eqref{Pbessel}
\begin{equation}\label{Gpbessel}
    V(r) =\frac{\nu^2-\frac{1}{4}}{r^2},\text{ for } r<1.
\end{equation}
 Assume $V\in C^{\infty}[0,\infty)$ and there is some constant $a\geq 0$ such that
 \begin{equation}\label{Gcons}
  V-a\in L^2[1,\infty).
 \end{equation}
Since  $-D^2+V$ is  unitarily equivalent to a component of  $-\Delta$
on a non-compact  manifold,   it is non-negative and $0$ is not an eigenvalue.

Assumption (\ref{Gcons}) will be easily satisfied by our construction.
 Actually, we will prove $|V(r)-\frac{|K_0|}{4}(n-1)^2|\leq
 \frac{h(r)}{1+r},$ therefore $a=\frac{|K_0|}{4}(n-1)^2$ works.

 In this section, we will set up a generalized eigenfunction expansion for Schr\"odinger operator \eqref{Gop0}.
   $L$ given by \eqref{Gop0} is a Bessel differential operator for $r<1$.
   $V$ has two singular points: $r=0$ and $r=\infty$.  Since  $\nu>1$ by \cite[Theorems X.10]{Simmp2}, $L$ is in the limit point case at $0$, and since  $V-a\in L^2[1,\infty)$, by \cite[Theorems X.28]{Simmp2}  $L$ is in the limit point case at $\infty$. So by Weyl's criterion,   $L$ is  essentially self-adjoint  on $C_0^{\infty}(0,\infty)$.

  Let us consider the eigen-equation
   \begin{equation}\label{equ1}
    Lu=z u,
   \end{equation}
   with $z\in \mathbb{C}$ and $z\neq 0$.
Let $u(r)=\sqrt{r}y(r)$. \eqref{equ1} becomes
 \begin{equation}\label{equ1may101}
    y^{\prime\prime}(r)+\frac{y^{\prime}(r)}{r}+(z-\frac{\nu^2}{r^2})y(r)=0.
   \end{equation}
   Let $\sqrt{z}r=x$.
   \eqref{equ1may101} becomes
 \begin{equation}\label{equ1may102}
    y^{\prime\prime}(x)+\frac{y^{\prime}(x)}{x}+(1-\frac{\nu^2}{x^2})y(x)=0.
   \end{equation}
   \eqref{equ1may102} is a standard Bessel equation and it has a solution $y(x)=J_{\nu}(x)$ (see e.g.  Chapter 17 in \cite{Modern}),
   where
   \begin{equation*}
    J_{\nu}(x)=\sum_{n=0}^{\infty}\frac{(-1)^n}{\Gamma(n+1)\Gamma(\nu+n+1)}\left(\frac{x}{2}\right)^{2n+\nu}.
   \end{equation*}
  Thus Bessel differential  equation \eqref{equ1}  has then a solution
   \begin{equation*}
    u(r)=\sqrt{r}J_{\nu}(\sqrt{z}r)
   \end{equation*}
   for $r<1$.
   It is easy to see that  $u\in L^2((0,1))$ and  since $L$ is in the limit point  case at $0$, it is unique  up to a normalization constant.

   Now we extend the solution $u$   to $r\geq1$ with  $u$ still solving \eqref{equ1}.
   For convenience, denote
   \begin{equation}\label{Gaug1}
   \tilde{J}_{\nu}(r,z):=u(r).
   \end{equation}
   We emphasise that
   \begin{equation}\label{Gaug2}
   \tilde{J}_{\nu}(r,z)= \sqrt{r}{J}_{\nu}(\sqrt{z}r)
   \end{equation}
   for $r<1$.
  Thus  $ \tilde{J}_{\nu}(z,r)$ is the   unique eigen-solution of \eqref{equ1}, such that
   $\tilde{J}_{\nu}(z,r)\in L^2((0,1)) $. Notice that  $ \tilde{J}_{\nu}(z,r) $ may be not in
   $L^2([1,\infty))$.

   Our main result in this section is
   \begin{theorem}\label{Thmexp}
   Suppose $V$ satisfies \eqref{Gpbessel} and \eqref{Gcons}. Assume $-D^2+V$ is a non-negative operator and $0$ is not an eigenvalue.
   Then there  exists a monotone   measurable function $\rho(\lambda)$ on $\mathbb{R}^+$ of locally bounded variation on $(0,\infty)$ such that the following statements hold,
   \begin{itemize}
   \item[I:] for any $f\in L^2(\R^+,dr)$ there exists a unique $\hat{f}\in L^2(\R^+,d\rho) $ such that
   \begin{equation*}
    \hat{f}(\lambda)=\int_{\mathbb{R}^+} f(r) \tilde{J}_{\nu}(r,\lambda) dr.
   \end{equation*}
   Conversely, for any $g\in   L^2(\R^+,d\rho)$, there exists a unique  $f\in L^2(\R^+,dr)$ such that $g=\hat{f}.$
   \item[II:] for any $f_1,f_2\in L^2(\R^+,dr)$, we have
   \begin{equation*}
    \int_{\mathbb{R}^+}f_1f_2 dr =\int_{\mathbb{R}^+}\hat{f}_1\hat{f}_2 d\rho.
   \end{equation*}
   \item[III:] for any $f \in L^2(\R^+,dr)$, let $g=\hat{f}$. Then  we have
   \begin{equation*}
    f(r) =\int_{\mathbb{R}^+} g(\lambda)  \tilde{J}_{\nu}(r,\lambda) d\rho(\lambda).
   \end{equation*}
   \item[IV:] Define the unitary operator  $U$ from $L^2(\R^+,dr)$ to $L^2(\R^+,d\rho)$ by
   \begin{equation*}
    Uf=\hat{f}
   \end{equation*}
   which is called the generalized Fourier transform.  Then we have
   $\hat{L}=ULU^{-1}$ is the multiplication operator on $L^2(\R^+,d\rho)$, that is,
    $$  \mathfrak{{D}}(\hat{L})=\{g:g(\lambda),\lambda g(\lambda) \in L^2(\R^+,d\rho) \}$$
   and
   \begin{equation*}
    (\hat{L} g)(\lambda) =\lambda g(\lambda)
   \end{equation*}
   for $g\in \mathfrak{D}(\hat{L})$.

   \end{itemize}
   \end{theorem}
 The proof is based on Titchmarsh expansion techniques  in
 \cite{titschmarch1958eigenfunction}. While they are rather standard,
 full details are needed to prove Theorem \ref{Thmexp}
in its full strength, so we list them here.
 We go over the classical Weyl theory  first.
 Suppose differential operator $T=-D^2+q$ on $L^2(\R^+)$ is in the limit point case on both sides $ 0$ and $\infty$. Thus $T$ is essentially  self-adjoint.
 We assume $z \in \mathbb{C}^+$.
 Let $\theta(x,z)$ and $\phi(x,z)$ be the solutions  of
 \begin{equation}\label{solution}
   \begin{cases}
-y^{\prime\prime}+qy=z y\\
y(1)=1\\
y^{\prime}(1)=0
\end{cases} \text{ and }
 \begin{cases}
-y^{\prime\prime}+qy=z y\\
y(1)=0\\
y^{\prime}(1)=1
\end{cases}.
 \end{equation}

Since both $0$ and $\infty$ are limit points, for $\Im z>0$,  there exist unique  $M_{-}(z)$ and
$M_{+}(z)$ so that
\begin{equation}\label{Weyl1}
  \psi_1(x,z) = \theta(x,z)+M_{-}(z)\phi(x,z)\in L^2(0,1]
\end{equation}
and
\begin{equation}\label{Weyl2}
   \psi_2(x,z) =\theta(x,z)+M_{+}(z)\phi(x,z)\in L^2[1,\infty).
\end{equation}

By the Weyl theory \cite[Formula 2.18.3]{titschmarch1958eigenfunction},
 we have
 \begin{equation}\label{Gweyl1}
 \int_0^1|\varphi(t,z)+M_{-}(z)\psi(t,z)|^2dt=\frac{\Im M_{-}(z)}{\Im z}
 \end{equation}
 and
 \begin{equation}\label{Gweyl2}
   \int_1^{\infty} | \varphi(t,z)+M_{+}(z)\psi(t,z)|^2 dt=-\frac{\Im M_{+}(z)}{\Im z}.
 \end{equation}
 Let
\begin{eqnarray}
   \label{M11} M_{11}(z) &=& -\frac{1}{M_{-}(z)-M_{+}(z)}, \\
  \label{M12}M_{12}(z) &=&  M_{21}(z)=-\frac{M_{-}(z)}{M_{-}(z)-M_{+}(z)}, \\
 \label{M22} M_{22}(z) &=&- \frac{M_{-}(z)M_{+}(z)}{M_{-}(z)-M_{+}(z)} .
\end{eqnarray}
All of $M_{jk}$, $j,k=1,2$ are Herglotz functions from $\mathbb{C}^+$ to $\mathbb{C}^+$.

Thus we can define monotone functions  $\rho_{jk}$ , $j,k=1,2$, (with locally bounded variation on $(-\infty,\infty)$, see p.58 in  \cite{titschmarch1958eigenfunction}) such that
\begin{equation*}
 \frac{1}{\pi}   M_{jk}(z)=\int_{\mathbb{R}}\frac{d\rho_{jk}(x)}{x-z}
\end{equation*}
for $\Im z>0$. Each
$\rho_{jk}(x)$ is unique up to a constant. Let $z=x+iy$ with $y>0$.
Then (formula 3.5.3 in p.58 of \cite{titschmarch1958eigenfunction})
\begin{equation}\label{Ginverse}
  \rho_{jk}(u_2)-\rho_{jk}(u_1)=\lim_{y\to 0}\frac{1}{\pi}\int_{u_1}^{u_2}\Im(M_{jk}(x+iy))dx.
\end{equation}
Denote  by $\rho$   the  matrix with coefficients $\rho_{jk}$, $j,k=1,2$ and let
\begin{equation*}
L^2_{\rho}=\left\{ g=\left(\begin{array}{c}
                     g_1 \\
                     g_2
                   \end{array}\right): ||g||_{\rho}^2=\int_{\mathbb{R}}\sum_{j,k=1}^2g_jg_k d\rho_{jk}(\lambda)<\infty
\right\}.
\end{equation*}
The inner product on $L^2_{\rho}$ is given by
\begin{equation*}
    (g,h)_{\rho}=\int_{\mathbb{R}}\sum_{j,k=1}^2g_jh_k d\rho_{jk}(\lambda).
\end{equation*}

\begin{theorem}\cite[formulas 3.1.8-3.1.11 in Chapter 3]{titschmarch1958eigenfunction}\label{ThmTchexp1}

  The following statements hold,
   \begin{itemize}
   \item[I:] for any $f\in L^2(\R^+,dx)$ there exists a unique $g=\left(\begin{array}{c}
                                                                          g_1 \\
                                                                          g_2
                                                                        \end{array}\right)
\in L^2_{\rho} $ such that
   \begin{equation}\label{g1}
    g_1(\lambda)=\int_{\mathbb{R}^+} f(x) \theta(x,\lambda)  dx,
   \end{equation}
and
 \begin{equation}\label{g2}
    g_2(\lambda)=\int_{\mathbb{R}^+} f(x) \phi(x,\lambda)  dx.
   \end{equation}
Denote $g=\hat{f}$ for simplicity.
   Conversely, for any
$g=\left(\begin{array}{c}
                                                                          g_1 \\
                                                                          g_2
                                                                        \end{array}\right)
\in L^2_{\rho} $,  there exists a unique  $f\in L^2(\R^+,dx)$ such that
$g=\hat{f}$.
   \item[II:] for any $f_1,f_2\in L^2(\R^+,dx)$, we have
   \begin{equation*}
    \int_{\mathbb{R}^+}f_1f_2 dx = (\hat{f}_1,\hat{f}_2)_{\rho}.
   \end{equation*}
   \item[III:] for any $f \in L^2(\R^+,dx)$, let $g=\left(\begin{array}{c}
                                                                          g_1 \\
                                                                          g_2
                                                                        \end{array}\right)=\hat{f}$. Then we have
   \begin{eqnarray}\label{Inversefourier}
      \label{Inversefourier}   f(x) &=& \int_{\mathbb{R}}\theta (x,\lambda)g_1(\lambda)d\rho_{11}(\lambda)+\theta (x,\lambda)g_2(\lambda)d\rho_{12}(\lambda) \\
                                                  \nonumber                                  && + \int_{\mathbb{R}}\phi (x,\lambda)g_1(\lambda)d\rho_{21}(\lambda)+\phi (x,\lambda)g_2(\lambda)d\rho_{22}(\lambda)
                                                                               \end{eqnarray}

   \item[IV:] Define the unitary operator  $U$ from $L^2(\R^+,dx)$ to $L^2_{\rho} $ by
   \begin{equation*}
    Uf=\hat{f}
   \end{equation*}
   which is called the generalized Fourier transform.  Then we have
   $\hat{L}=ULU^{-1}$ is the multiplication operator on $L^2_{\rho} $, that is,
      $$  \mathfrak{{D}}(\hat{L})=\{g:g(\lambda),\lambda g(\lambda) \in L^2_{\rho} \}$$
   and
   \begin{equation*}
    (\hat{L} g)(\lambda) =\lambda g(\lambda)
   \end{equation*}
   for $g\in \mathfrak{D}(\hat{L})$.

   \end{itemize}
   \end{theorem}
   \begin{theorem}\cite[formula 3.1.12 in Chapter 3]{titschmarch1958eigenfunction}\label{ThmTchexp}
   Suppose $\lim_{\Im z\to 0}\Im M_-(z)=0$.
  Then
  \begin{equation}\label{Grela}
    d\rho_{12}(\lambda)=M_-(\lambda) d\rho_{11}(\lambda), d\rho_{22}(\lambda)=M^2_-(\lambda) d\rho_{11}(\lambda).
  \end{equation}
  Moreover,
  the  following statements hold,
   \begin{itemize}
   \item[I:] for any $f\in L^2(\R^+,dr)$ there exists a unique $\hat{f}\in L^2(\R^+,d\rho_{11}) $ such that
   \begin{equation*}
    \hat{f}(\lambda)=\int_{\mathbb{R}^+} f(r) \psi_1(r,\lambda) dr.
   \end{equation*}
   Conversely, for any $g\in   L^2(\R,d\rho_{11})$, there exists a unique  $f\in L^2(\R^+,dr)$ such that $g=\hat{f}.$
   \item[II:] for any $f_1,f_2\in L^2(\R^+,dr)$, we have
   \begin{equation*}
    \int_{\mathbb{R}^+}f_1f_2 dr =\int_{\mathbb{R}}\hat{f}_1\hat{f}_2 d\rho_{11}.
   \end{equation*}
   \item[III:] for any $f \in L^2(\R^+,dr)$, let $g=\hat{f}$. Then  we have
   \begin{equation*}
    f(r) =\int_{\mathbb{R}} g(\lambda)  \psi_1( \lambda,r) d\rho_{11}(\lambda).
   \end{equation*}
   \item[IV:] Define the unitary operator  $U$ from $L^2(\R^+,dr)$ to $L^2(\R,d\rho_{11})$ by
   \begin{equation*}
    Uf=\hat{f}
   \end{equation*}
   which is called the generalized Fourier transformation.  Then we have
   $\hat{L}=ULU^{-1}$ is the multiplication operator on $L^2(\R^+,d\rho)$, that is,
    $$  \mathfrak{{D}}(\hat{L})=\{g:g(\lambda),\lambda g(\lambda) \in L^2(\R,d\rho_{11}) \}$$
   and
   \begin{equation*}
    (\hat{L} g)(\lambda) =\lambda g(\lambda)
   \end{equation*}
   for $g\in \mathfrak{D}(\hat{L})$.

   \end{itemize}

   \end{theorem}
   We remark that $\psi_1(r,\lambda)$ is given by \eqref{Weyl1} and $\rho_{11}$ is given by \eqref{Ginverse}.
\begin{proof}[\bf Proof of Theorem \ref{Thmexp}]
We will use Theorems \ref{ThmTchexp1} and \ref{ThmTchexp} to prove Theorem \ref{Thmexp}.
Applying Theorem \ref{ThmTchexp} to operator \eqref{equ1}, we obtain $M_{jk}$ and $\rho_{jk}$.
By the assumption that $-D^2+V$ is non-negative and $0$ is not an eigenvalue,
$d\rho_{jk}(\lambda)$ is supported  on $(0,\infty)$, for $j,k=1,2$.

Recall that   $\tilde{J}_{\nu}(z,r)$, $z\in \mathbb{C}$, is the unique solution of \eqref{equ1}   in $L^2(0,1]$.
Thus one has for $r>0$,
\begin{equation}\label{equ2}
  \psi_1(r,z)= \theta(r,z)+M_{-}(z)\phi(r,z) =C\tilde{J}_{\nu}(r,z).
\end{equation}

Let $r=1$ in \eqref{equ2}, using the boundary condition of $  \theta,\phi$ at $r=1$ and $\tilde{J}_{\nu}(r,z)=\sqrt{r}J_{\nu}(\sqrt{z}r)$, one has
\begin{equation*}
    C=\frac{1}{J_{\nu}(\sqrt{z})},
\end{equation*}
and
\begin{equation}\label{equ3}
  M_{-}(z)=\frac{\tilde{J}^{\prime}_{\nu}(1,z)}{\tilde{J}_{\nu}(1,z)}=\frac{1}{2}+\frac{\sqrt{z}}{J_{\nu}(\sqrt{z})}{J_{\nu}^{\prime}(\sqrt{z})},
\end{equation}
for $\Im z>0$.
It implies
\begin{equation}\label{Gpsi1}
   \psi_1(r,z)= \frac{\tilde{J}_{\nu}(r,z)}{J_{\nu}(\sqrt{z})}.
\end{equation}

Thus $ M_{-}(z) $ can be extended to $\mathbb{R}$ except for the zeros of $J_{\nu}(\sqrt{z})$, that is
\begin{equation}\label{equ4}
  M_{-}(z)=\frac{1}{2}+\frac{\sqrt{z}}{J_{\nu}(\sqrt{z})}{J_{\nu}^{\prime}(\sqrt{z})},
\end{equation}
for  $\Im z\geq 0$. Moreover,
\begin{equation}\label{equ5}
 \Im M_{-}(\lambda)=0,
\end{equation}
for $\lambda\geq 0$ and $J_{\nu}(\sqrt{\lambda})\neq 0$.

Let
\begin{equation}\label{Grho}
  d {\rho}(\lambda)=\frac{1}{J_{\nu}^2(\sqrt{\lambda})}d\rho_{11}(\lambda).
\end{equation}

By \eqref{Gpsi1} and Theorem \ref{ThmTchexp}, we obtain Theorem
\ref{Thmexp} except for  the local   boundedness of variation of $ \rho$ on
$(0,\infty)$. To prove the latter, fix $[a,b]\subset (0,\infty)$.
For any given $\lambda_0\in [a,b]$,  $ \rho$ is of bounded  variation in a neighborhood of $\lambda_0$ if $J_{\nu}(\sqrt{\lambda_0})\neq 0$.
Suppose $J_{\nu}(\sqrt{\lambda_0})= 0 $. It is easy to see that  $\tilde{J}^{\prime}_{\nu}(1,\lambda_0)\neq 0$ so that  $|\tilde{J}^{\prime}_{\nu}(1,\lambda)|>\delta>0$ for $\lambda\in(\lambda_0-\epsilon,\lambda_0+\epsilon)$ with some $\epsilon,\delta>0$.

By \eqref{Grela} and \eqref{equ3}, one has
 \begin{equation*}
 d\rho_{22}=\frac{(\tilde{J}^{\prime}_{\nu}(1,\lambda))^2}{J^2_{\nu}(\sqrt{\lambda})}d\rho_{11}.
 \end{equation*}
 Thus
 \begin{equation*}
    d\rho= \frac{1}{(\tilde{J}^{\prime}_{\nu}(1,\lambda))^2} d\rho_{22}.
 \end{equation*}
 By the fact that $\rho_{22}$ is of  bounded variation   on $[a.b]$
 and  $|\tilde{J}^{\prime}_{\nu}(1,\lambda)|>\delta>0$ for
 $\lambda\in(\lambda_0-\epsilon,\lambda_0+\epsilon)$, we have that $\rho$
 is of bounded variation   on
 $\lambda\in(\lambda_0-\epsilon,\lambda_0+\epsilon)$. Since there are
 finitely many zeros of $J_{\nu}(\sqrt{\lambda})$ in $[a,b]$, this
 completes the proof.

\end{proof}
\begin{lemma}\label{el}
Under the condition of Theorem \ref{Thmexp}, suppose $\lambda_0>0$  is an eigenvalue.
  Then $\rho(\lambda_0)= \|\tilde{{J}}_{\nu}(\cdot,\lambda_0)\|_{L^2(\mathbb{R}^+)}^{-2}.$
\end{lemma}
\begin{proof}

Suppose $\lambda_0$ is an eigenvalue.
 Then  $f=\tilde{{J}}_{\nu}(r, \lambda_0)$ is the  corresponding eigenfunction.
The Fourier transform  $\hat{f}$ of  $f$ is well defined, and
\begin{equation*}
     \hat{{f}}(\lambda)=\int_{\mathbb{R}^+}f(r)\tilde{{J}}_{\nu}(r, \lambda) dr=\int_{\mathbb{R}^+}\tilde{{J}}_{\nu}(r, \lambda_0)\tilde{{J}}_{\nu}(r, \lambda) dr.
\end{equation*}
It leads to
\begin{equation}\label{Gmay201}
    \hat{{f}}(\lambda_0)=\int_{\mathbb{R}^+}\tilde{{J}}^2_{\nu}(r, \lambda_0) dr.
\end{equation}
By the fact that $\lambda_0$ is an eigenvalue and Theorem   \ref{Thmexp}, one has
\begin{equation*}
    \hat{f}= C\chi_{\{\lambda_0\}}(\lambda)
\end{equation*}
where $C $ is a constant and $\chi_{\{\lambda_0\}}(\lambda)$  is the characteristic  function of  $\lambda_0$.
It implies
\begin{equation}\label{Gmay202}
   | |\hat{f}||^2_{\rho}=|\hat{f}(\lambda_0)|^2\rho(\lambda_0),
\end{equation}
where $| |\hat{f}||^2_{\rho}=(\int_{\R^+}|\hat{f}(\lambda)|^2d\rho)^{1/2}$.
By II of Theorem   \ref{Thmexp},
 we have
\begin{equation}\label{Gmay203}
   ||\hat{f}||^2_{\rho}=\int_{\mathbb{R}^+}f(r)^2 dr=\int_{\mathbb{R}^+}\tilde{{J}}^2_{\nu}(r, \lambda_0) dr.
\end{equation}
Now the Lemma follows from \eqref{Gmay201}, \eqref{Gmay202} and \eqref{Gmay203}.
\end{proof}

 \section{Preparations}

In this section, we will use Kiselev's construction \cite{kiselev2005imbedded} to prove
 \begin{theorem}\label{thmke}
Fix any $\tau\geq0$, $0<\delta<1/2$, $\nu>1$ and positive function $h(r)$ such that $\lim_{r\to \infty}h(r)=\infty$.
Suppose $\tilde{V}(r)\in C^{\infty}(0,b)$ with $b\geq 10$ and for $r\leq 1$, $\tilde{V}(r)=\frac{\nu^2-\frac{1}{4}}{r^2}$.  Then
there exist  a potential $V(r)$  on $(0,\infty)$ satisfying  the following statements:

\begin{itemize}
\item [I.] $V(r)=\tilde{V}(r)$ for $0<r\leq b-\delta$.
  \item [II.] $V(r)\in C^{\infty}(0,\infty)$  and $|V(r)-\tau^2|\leq \frac{h(r)}{1+r}$ for $r\geq b+\delta$.
  \item  [III.] $\sup_{b-\delta\leq r\leq b+\delta}|V(r)|\leq \tau^2+1+\sup_{b-1\leq r\leq b}\tilde{V}(r)|$.
  \item [IV.] $-D^2+V$ has singular continuous spectrum on $[\tau^2,\infty)$.
\end{itemize}
\end{theorem}
 Let $u$ be a solution of  $-D^2+\frac{\nu^2-\frac{1}{4}}{r^2}$, $r<1$, such that
 \begin{equation*}
    u\in L^2(0,1).
 \end{equation*}
 Recall that (by \eqref{Gaug1} and \eqref{Gaug2})  for $r<1$,
 \begin{equation*}
    u(r,\lambda)= \sqrt{r}{J}_{\nu}(\sqrt{\lambda}r).
 \end{equation*}
Let $\lambda=k^2$.  For
 any $k>\tau$, let
 \begin{equation*}
    \bar{k}=\sqrt{k^2-\tau^2}.
 \end{equation*}
 Set $V(r)=\frac{\nu^2-\frac{1}{4}}{r^2}$ for $r<1$. Suppose we construct potentials   $V(r)$ on $(0,x]$ .  We extend $u$  to $(0,x]$ by solving
 \begin{equation*}
   (-D^2+V)u(r)=\lambda u(r)
 \end{equation*}
 for $0\leq r\leq x$.

 It will be convenient to introduce the modified   Pr\"ufer variables
$R$ and $\theta,$ $R^2 = (u')^2 + \bar{k}^2 u^2$ and $\theta = \tan^{-1}(\bar{k}u/u')$ for $r\geq 1$.
 Then it is easy to see that for $r\geq 1$,
\begin{eqnarray}\label{pruferam}
(\log R(r,\bar{k})^2)' = \frac{1}{\bar{k}}({V}(r)-\tau^2) \sin 2\theta(r,\bar{k}), \\
\label{pruferan}
\theta(r,\bar{k})' = \bar{k} - \frac{1}{\bar{k}}({V}(r)-\tau^2) (\sin \theta)^2.
\end{eqnarray}

\begin{proof}[\bf Proof of Theorem \ref{thmke}]
The proof of Theorem \ref{thmke} closely follows the construction of
\cite[Theorem 1.1]{kiselev2005imbedded}, so we skip the details.
We point out several  small modifications.
\begin{itemize}
  \item Replace Lemma 2.1 in \cite{kiselev2005imbedded}  with Lemma \ref{el}. Replace the Pr\"ufer variables (2.2) and (2.3) in \cite{kiselev2005imbedded} with \eqref{pruferam} and \eqref{pruferan}.
  \item  I and III  follow from Theorem 1.1 in \cite{kiselev2005imbedded}.
  \item  The potential  constructed  in  \cite{kiselev2005imbedded} is not smooth. This issue can be addressed in the following way. In \cite{kiselev2005imbedded}, Kiselev  constructed  the potential $V$  piece by piece.  We  need to smooth the potential for the current piece first and then construct the next piece.
   II  comes from the fact that we  smooth the potential around $r=b$.
\end{itemize}
\end{proof}

Without loss of  generality, assume  $h(r)$ is positive,  non-decreasing, $\lim_{r\to\infty} h(r)=\infty$ and
\begin{equation*}
  h(r)\leq  1+r^{1/10}.
\end{equation*}
In the following $b$ is a large positive constant, and $\delta$ is a small positive constant. We will need the following Lemma.

\begin{lemma}[Comparison theorem] \label{lecom1}
Suppose $f(r)\geq g(r)$ for $2\leq r\leq r_0$. Let us consider two differential equations for $r\geq r_0$,
\begin{equation}\label{Gode1}
    f^{\prime}+m(r)f^2(r)+\frac{Ae^{(n-1)r}}{\exp(\int_2^r2(1 +f(x)e^{-(n-1)x})dx)}=h_1(r)
\end{equation}
and
\begin{equation}\label{Gode2}
    g^{\prime}+m(r)g^2(r)+\frac{Ae^{(n-1)r}}{\exp(\int_2^r2(1+g(x)e^{-(n-1)x})dx)}=h_2(r),
\end{equation}
where  A is a non-negative constant and $m(r)\geq 0$.
Suppose $f(r_0)\geq g(r_0)$ and $h_1(r)\geq h_2(r)$ for all $r>r_0$.
Then $f(r)\geq g(r)$ for all possible $r\geq r_0$.
\end{lemma}
\begin{proof}
Suppose $f(r)\geq g(r)$ for $r_0<r<r_1$ and  $f(r_1)=g(r_1)$.
Since $f(r)\geq g(r)$ for $2<r\leq r_0$, one has
\begin{equation*}
\frac{Ae^{(n-1)r_1}}{\exp(\int_2^{r_1}2(1 +f(x)e^{-(n-1)x})dx)}\leq \frac{Ae^{(n-1)r_1}}{\exp(\int_2^{r_1}2(1+g(x)e^{-(n-1)x})dx)}.
\end{equation*}
By \eqref{Gode1}, \eqref{Gode2} and $h_1(r_1)\geq h_2(r_1)$, we have
\begin{equation*}
  f'(r_1)\geq g'(r_1).
\end{equation*}
It implies the lemma.
\end{proof}

\section{Proof of Theorem \ref{Thmmanifold}}

We plan to use rotationally symmetric metric to complete our construction.
Our objective is to construct proper $f_1(r)$ so that Riemannian manifold $$(M_n, g)
     = \Bigl( \R^+ \times S^{n-1}(1)\cup \{O\}
     , dr^2 + f_1^{\,2}(r)g_{S^{n-1}(1)} \Bigr)$$ satisfies Theorem \ref{Thmmanifold}.
In the neighbourhood of the origin, we will use the Euclidean metric. For   $r\geq 2$.
  we will construct  $f(r)$  so that
  \begin{equation}\label{Gf11}
    f_1(r)=\exp(\int_2^r (\sqrt{|K_0|}+f(x))dx),
  \end{equation}
will have the desired properties.

     Define
     \begin{align}
     K_{\rm rad}(r)   := & -\frac{f_1''(r)}{f_1(r)},\label{C}\\
     q(r) := & \frac{(n-1)(n-3)}{4}\left(\frac{f_1'(r)}{f_1(r)}\right)^2 - \frac{(n-1)}{2}K(r)+\frac{\lambda_i}{f_1^2}\label{D}.
\end{align}
Direct computation yields that for $r\geq2$
\begin{align}
     K_{\rm rad}(r)
     = &-  |K_0| -2\sqrt{|K_0|}f(r)-f^2(r)-f^{\prime}(r),\label{DeKr}\\
     q(r)
     = & \frac{(n-1)^2}{4}(\sqrt{|K_0|}+f(r))^2+\frac{n-1}{2}f^{\prime}(r)+\frac{\lambda_i}{f_1^2}.\label{Deq_0r}
\end{align}
Let
\begin{equation}\label{Gtau}
  \tau=\frac{(n-1)}{2}\sqrt{|K_0|}.
\end{equation}
In order to prove  Theorem \ref{Thmmanifold},   we need to show that there exists
  $f$ such that
 \begin{equation}\label{Solf}
    q(r)=V(r)
 \end{equation}
 and
 $K_{\rm rad}(r)$ given by \eqref{DeKr} satisfies our  goal, where $V(r)$  is given by Theorem \ref{thmke}.

Without loss of generality, we assume $K_0=-1$.

\begin{proof}[\bf Proof of Theorem \ref{Thmmanifold}]

For $r\leq 1$, let
  \begin{equation}\label{Gf10}
    \tilde{f}_1(r)= r.
  \end{equation}
   For $r\in[2,b]$, let
  \begin{equation}\label{Gf10}
    \tilde{f}_1(r)= e^{2(r-2)}.
  \end{equation}
We extend $ \tilde{f}_1(r)$ to $(0,b)$  so that $ \tilde{f}_1(r)>0$  and $ \tilde{f}_1\in C^{\infty}(0,b]$.
Let $\tilde{V}(r)$ be given by \eqref{D} for $r\leq b$, namely,
\begin{equation*}
\tilde{V}(r)=\frac{(n-1)(n-3)}{4}\left(\frac{\tilde{f}_1^{\prime}}{\tilde{f}_1} \right)^2+\frac{(n-1)}{2} \frac{\tilde{f}_1^{\prime\prime}}{\tilde{f}_1}+\frac{\lambda_i}{\tilde{f}_1^2}.
\end{equation*}
In particular,
for $0<r\leq 1$,
\begin{equation*}
\tilde{V}(r)=\left(\frac{(n-1)(n-3)}{4}+\lambda_i\right)\frac{1}{r^2}=\frac{\nu^2-\frac{1}{4}}{r^2}.
\end{equation*}

By Theorem \ref{thmke}, we obtain a potential $V(r)$. Now we are ready to define our metric.
For $ r\leq b-\delta$, let
\begin{equation*}
  f_1(r)=\tilde{f}_1(r).
\end{equation*}
For $2\leq r\leq b-\delta$, let
  $f(r)=1$ so that
for  $ 2\leq r\leq b-\delta$,
\begin{equation*}
  f_1(r)=\exp(\int_2^r (1+f(x))dx).
\end{equation*}
Let us consider the following equation ($n\geq 2$)
\begin{equation}\label{Solf1new}
    \tau^2+\frac{ (n-1)^2}{2}f+ \frac{(n-1)^2}{4}f^2+\frac{n-1}{2}f^{\prime}+\frac{\lambda_i}{\exp(\int_2^r 2(1+f(x))dx)}=V(r).
\end{equation}
Since  $V$ is defined on $(0,\infty)$ and $f\equiv 1$   on $(2,b-\delta)$, let $f(r)$ solve \eqref{Solf1new} with initial condition $f(b-\delta)=1$.  By choosing $\delta$ sufficiently small and III of Theorem \ref{thmke}, there is a unique solution $f(r)$  for $b-\delta\leq r\leq b+\delta$ such that $1/2<f(r)<2$ for $b-\delta\leq r\leq b+\delta$.
Let $f(r)$ solve the   equation \eqref{Solf1new}
for $r\geq b+\delta$. We claim (see Lemma \ref{Lesolh} below) that there exists a unique solution $f(r)$ for all $r\geq b+\delta$ such that
\begin{equation}\label{Anew}
  |f(r)|+|f'(r)|= O\left( \frac{h(r)}{1+r}\right).
\end{equation}
For $r>2$, define
\begin{equation*}
  f_1(r)=\exp(\int_2^r (1+f(x))dx).
\end{equation*}
By our construction,  for $r>0$,
\begin{equation}\label{E}
 V(r)= \frac{(n-1)(n-3)}{4}\left(\frac{f_1^{\prime}}{f_1}\right)^2+\frac{n-1}{2}\frac{f_1^{\prime\prime}}{f_1}+\frac{\lambda_i}{f_1^2}.
\end{equation}
By \eqref{DeKr} and \eqref{Anew}, one has
\begin{equation*}
    |K_{\rm rad}(r)+1|=  O\left( \frac{h(r)}{1+r}\right).
\end{equation*}
By Theorem \ref{thmke}, $-D^2+V$ has non-empty singular continuous
spectrum. By \eqref{Gp} and  \eqref{E},  we have that $ -\Delta$ also has non-empty singular continuous spectrum.
By \eqref{Gf11}, one has
\begin{equation*}
    \lim_{r\to\infty} \Delta r = \lim_{r\to\infty}(n-1)\frac{f_1'(r)}{f_1(r)}= (n-1)+\lim_{r\to\infty}(n-1)f(r)=(n-1) .
\end{equation*}
By Theorem 1.2 in \cite{kumura1997essential},
we have
$$\sigma_{{\rm ess}}(-\Delta)= \left[\frac{(n-1)^2}{4},\infty  \right).$$
Thus
\begin{equation*}
    \sigma_{{\rm ac}}(-\Delta)\subset \left[\frac{(n-1)^2}{4},\infty  \right).
\end{equation*}
By the fact that  for $r>b+\delta$
\begin{equation*}
   \left| V(r)-\frac{(n-1)^2}{4}\right|\leq \frac{h(r)}{1+r},
\end{equation*}
one has  (e.g. \cite{kis1,christ1998absolutely})
\begin{equation*}
    \left[\frac{(n-1)^2}{4},\infty  \right)\subset \sigma_{{\rm ac}}(-\Delta).
\end{equation*}
Thus
\begin{equation*}
    \sigma_{{\rm ac}}(-\Delta)= \left[\frac{(n-1)^2}{4},\infty  \right).
\end{equation*}
which completes the proof.

\end{proof}

\begin{lemma}\label{Lesolh}
Suppose  for $ 2\leq r\leq b-\delta$,  $ f(r)=1$ and for  $ b-\delta\leq r\leq b+\delta$, $\frac{1}{2}\leq f(r)\leq 2$.
Suppose $|V(r)-\tau^2|\leq \frac{h(r)}{1+r}$ for $r\geq b+\delta$.
Let $f(r)$ solve the following equation ($n\geq 2$)
\begin{equation}\label{Solf1}
    \tau^2+\frac{ (n-1)^2}{2}f+ \frac{(n-1)^2}{4}f^2+\frac{n-1}{2}f^{\prime}+\frac{\lambda_i}{\exp(\int_2^r 2(1+f(x))dx)}=V(r),
\end{equation}
for $r\geq b+\delta$. Then there exists a unique solution $f(r)$ for all $r\geq b+\delta$ such that
\begin{equation}\label{lasth}
  |f(r)|=  O\left( \frac{h(r)}{1+r}\right),
\end{equation}
and
\begin{equation}\label{A}
  |f'(r)|= O\left( \frac{h(r)}{1+r}\right).
\end{equation}
\end{lemma}
\begin{proof}
Let $r_0=b+\delta$.
Let
\begin{equation}\label{deft}
    f(r)=t(r)e^{-(n-1)r},
\end{equation}
and $t(r_0)=f(r_0)e^{(n-1)r_0}=t_0$.
Then for $r\geq r_0$,
$t(r)$ satisfies equation
\begin{equation}\label{last2}
    t^{\prime}+\frac{n-1}{2}t^2e^{-(n-1)r}+\frac{2\lambda_i}{n-1}\frac{e^{(n-1)r}}{\exp(\int_2^r2(1+t(x)e^{-(n-1)x})dx)}= g(r),
\end{equation}
where $g(r)=\frac{2}{n-1}({V}(r)-\tau^2)e^{(n-1)r}$. By the assumption, one has
\begin{equation}\label{B}
    |g(r)|\leq \frac{2}{n-1}\frac{h(r)}{1+r} e^{(n-1)r}.
\end{equation}
By a simple computation, one has for large $r_0$,
\begin{eqnarray}
 \nonumber \int_{r_0}^r \frac{h(x)}{1+x} e^{(n-1)x} dx&\leq & \frac{h(r)}{1+r}e^{(n-1)r}\int_{r_0}^r \frac{1+r}{1+x} e^{(n-1)(x-r)} dx\\
\label{last5}    &\leq & \frac{3h(r)}{1+r}e^{(n-1)r} .
\end{eqnarray}
By \eqref{last2},
 one has
\begin{equation}\label{last3}
   t(r)\leq t_0+\int_{r_0}^r |g(x)|dx\leq t_0+\frac{6h(r)}{1+r}e^{(n-1)r}.
\end{equation}
Now we will use Lemma \ref{lecom1} to get the lower bound of $t(r)$.
Let
\begin{equation*}
\hat{t}=- 10\int_{r_0}^r\frac{h(x)}{1+x} e^{(n-1)x} dx.
\end{equation*}
Let  $\hat{t}(r)=0$ for $2<r<r_0$.
Direct computation yields that
\begin{equation}\label{last2hat}
    \hat{t}^{\prime}+\frac{n-1}{2}\hat{t}^2e^{-(n-1)r}+\frac{2\lambda_i}{n-1}\frac{e^{(n-1)r}}{\exp(\int_2^r2(1+\hat{t}(x)e^{-(n-1)x})dx)}= \hat{g}(r),
\end{equation}
where
\begin{eqnarray*}
  \hat{g}(r) &=& -10\frac{h(r)}{1+r} e^{(n-1)r}+O(1)\frac{e^{(n-1)r}}{\exp(\int_2^r2(1+\hat{t}(x)e^{-(n-1)x})dx)} \\
   &&-O(1)e^{-(n-1)r}\left(\int_{r_0}^r\frac{h(x)}{1+x} e^{(n-1)x} dx\right)^2\\
  &\leq& -10\frac{h(r)}{1+r} e^{(n-1)r}+O(1)\frac{e^{(n-1)r}}{ e^{2r}}-O(1)e^{(n-1)r}\frac{h^2(r)}{(1+r)^2}  \\
     &\leq &  -10\frac{h(r)}{1+r} e^{(n-1)r}+O(1)e^{(n-3)r}-O(1)e^{(n-1)r}\frac{h^2(r)}{(1+r)^2} ,
\end{eqnarray*}
where the first  inequality holds  by \eqref{last5}. 

By  \eqref{B} and choosing large $r_0$, one has
$\hat{g}(r)\leq g(r)$ for  $r\geq r_0$.
 By  Lemma \ref{lecom1} and \eqref{last2}, one has
\begin{equation}\label{last7}
    t(r)\geq \hat{t}(r)\geq - 10\frac{h(r)}{1+r}e^{(n-1)r}.
\end{equation}
By \eqref{last3} and \eqref{last7}, we obtain that
\begin{equation}\label{Gjuly111}
   | t(r)|\leq t(r_0)+10\frac{h(r)}{1+r}e^{(n-1)r}.
\end{equation}
It implies \eqref{lasth}.  \eqref{B} follows from  \eqref{Solf1} and \eqref{lasth}.
\end{proof}
 \section*{Acknowledgments}
    This research was
 supported by NSF DMS-1401204, DMS-1901462,  and  DMS-1700314.


\end{document}